\documentclass[12pt,a4paper]{amsart}
\usepackage{amsfonts, amssymb, amsmath, amsthm}
\usepackage{graphicx,latexsym,eufrak,mathrsfs}
\usepackage[margin=1.2in]{geometry}

\newtheorem{theorem}{Theorem}[section]
\newtheorem{proposition}[theorem]{Proposition}
\newtheorem{lemma}[theorem]{Lemma}
\newtheorem{corollary}[theorem]{Corollary}

\theoremstyle{definition}

\theoremstyle{remark}
\newtheorem{remark}[theorem]{Remark}

\numberwithin{equation}{section}

\DeclareMathAlphabet{\mathpzc}{OT1}{pzc}{m}{it}

\begin{document}

\title[Ultradistributional boundary values of harmonic functions] { Ultradistributional boundary values of harmonic functions on the sphere}

\author[\DJ.Vu\v{c}kovi\'{c}]{\DJ or\dj e Vu\v{c}kovi\'{c}}
\address{Department of Mathematics, Ghent University, Krijgslaan 281, 9000 Ghent, Belgium}
\email{dorde.vuckovic@UGent.be}

\author[J. Vindas]{Jasson Vindas}
\address{Department of Mathematics, Ghent University, Krijgslaan 281, 9000 Ghent, Belgium}
\email{jasson.vindas@UGent.be}

\thanks{The authors gratefully acknowledge support by Ghent University, through the BOF-grant 01N01014.}

\subjclass[2010]{Primary 31B05, 46F20; Secondary 42C10, 46F15.}
\keywords{Harmonic functions on the unit ball; boundary values on the sphere; partial derivatives of spherical harmonics; support of ultradistributions; ultradifferentiable functions; Abel summabiity.}

\maketitle

\begin{abstract} We present a theory of ultradistributional boundary values for harmonic functions defined on the Euclidean unit ball. We also give a characterization of ultradifferentiable functions and ultradistributions on the sphere in terms of their spherical harmonic expansions. To this end, we obtain explicit estimates for partial derivatives of spherical harmonics, which are of independent interest and refine earlier estimates by Calder\'{o}n and Zygmund. 
We apply our results to characterize the support of ultradistributions on the sphere via Abel summability of their spherical harmonic expansions.
\end{abstract}

\section{Introduction}
The study of boundary values of harmonic and analytic functions is a classical and important subject in distribution and ultradistribution theory. There is a vast literature dealing with boundary values on $\mathbb{R}^{n}$, see e.g. \cite{alvarez2007,C-K-P,C-M,D-P-VBV2015,f-g-g,Hormander,Petzsche84} and references therein. 
In the case of the unit sphere $\mathbb{S}^{n-1}$, the characterization of harmonic functions in the Euclidean unit ball of $\mathbb{R}^{n}$ having distributional boundary values on $\mathbb{S}^{n-1}$ was given by Estrada and Kanwal in \cite{Estrada}.
In a recent article \cite{GonzalezV2016}, Gonz\'{a}lez Vieli has used the Poisson transform to obtain a very useful description of the support of a Schwartz distribution on the sphere (cf. \cite{V-E2010} for support characterizations on  $\mathbb{R}^{n}$). Representations of analytic functionals on the sphere \cite{Morimoto1998} as initial values of solutions to the heat equation were studied by Morimoto and Suwa \cite{Morimoto2000}.

In this article we generalize the results from \cite{Estrada} to the framework of ultradistributions \cite{Komatsu,Komatsu2} and supply a theory of ultradistributional boundary values of harmonic functions on $\mathbb{S}^{n-1}$. Our goal is to characterize all those harmonic functions $U$, defined in the unit ball, that admit boundary values $
\lim_{r\to 1^{-}} U(r\omega)$
in an ultradistribution space ${\mathcal{E}^{\ast}}'(\mathbb S^{n-1})$. Our considerations apply to both non-quasianalytic and quasianalytic ultradistributions, and, in particular, to analytic functionals. As an application, we also obtain a characterization of the support of a non-quasianalytic ultradistribution in terms of Abel summability of its spherical harmonic series expansion. Since Schwartz distributions are naturally embedded into the spaces of ultradistributions in a support preserving fashion, our support characterization contains as a particular instance that of Gonz\'{a}lez Vieli quoted above.

In Section \ref{section spherical harmonics} we study spaces of ultradifferentiable functions and ultradistributions through spherical harmonics. Our main results there are descriptions of these spaces in terms of the decay or growth rate of the norms of the projections of a function or an ultradistribution onto the spaces of spherical harmonics. We also establish the convergence of the spherical harmonic series in the corresponding space. Note that eigenfunction expansions of ultradistributions on compact analytic manifolds have recently been investigated in \cite{D-R2014,D-R2016} with the aid of pseudodifferential calculus (cf. \cite{Dorde-V2016} for the Euclidean global setting). However, our approach here is quite different and is rather based on explicit estimates for partial derivatives of solid harmonics and spherical harmonics that are obtained in Section \ref{section: estimates spherical harmonics}.  
Such estimates are of independent interest and refine earlier bounds by Calder\'{on} and Zygmund from \cite{Calderon}.  

Harmonic functions with ultradistributional boundary values are characterized in Section \ref{section: boundary values harmonic}. The characterization is in terms of the growth order of the harmonic function near the boundary $\mathbb{S}^{n-1}$; we also show in Section \ref{section: boundary values harmonic} that a harmonic function satisfying such growth conditions must necessarily be the Poisson transform of an ultradistribution. In the special case of analytic functionals, our result yields as a corollary: \emph{any} harmonic function on the unit ball arises as the Poisson transform of some analytic functional on the sphere. Finally, Section \ref{section: support sphere} deals with the characterization of the support of non-quasianalytic ultradistributions on $\mathbb{S}^{n-1}$.

\section{Preliminaries}
We employ the notation $\mathbb B^n$ for the open unit ball of $\mathbb R^n$.  We work in dimension $n\geq 2$.

\subsection{Spherical harmonics}The theory of spherical harmonics is a classical subject in analysis and it is very well explained in several textbooks (see e.g. \cite{Han,Axler}). The space of solid harmonics of degree $j$ will be denoted by ${\mathcal H}_j(\mathbb R^n)$, its elements are the harmonic homogeneous polynomials of degree $j$ on $\mathbb{R}^{n}$.  A spherical harmonic of degree $j$ is the restriction to $\mathbb S^{n-1}$ of a solid harmonic of degree $j$ and we write $\mathcal H_j(\mathbb S^{n-1})$ for space of all spherical harmonics of degree $j$. Its dimension, denoted as $d_j=\operatorname*{dim} \mathcal{H}_{j}(\mathbb{S}^{n-1})$, is (cf. \cite{Axler} or \cite[Thm. 2, p. 117]{Se})

$$d_j=\frac{(2j+n-2)(n+j-3)!}{j!(n-2)!}\sim \frac{2j^{n-2}}{(n-2)!}.
$$
From this exact formula, it is not hard to see that $d_j$ satisfies the bounds
\begin{equation}\label{dj}
\frac{2}{(n-2)!} j^{n-2}<d_j\leq n j^{n-2}, \quad \mbox{for all } j\geq1.  
\end{equation}

It is well known \cite{Axler} that
$$
L^{2}(\mathbb{S}^{n-1})=\bigoplus_{j=0}^{\infty} \mathcal{H}_{j}(\mathbb{S}^{n-1}),
$$
where the $L^2$-inner product is taken with respect to the surface measure of $\mathbb{S}^{n-1}$. The orthogonal projection of $f\in L^{2}(\mathbb{S}^{n-1})$ onto $\mathcal{H}_{j}(\mathbb{S}^{n-1})$ will always be denoted as $f_{j}$; it is explicitly given by
$$
f_{j}(\omega)=\frac{1}{|\mathbb{S}^{n-1}|}\int_{\mathbb S^{n-1}} f(\xi) Z_j(\omega,\xi)d\xi, $$
where  $Z_j(\omega,x)$ is the zonal harmonic of degree $j$ with pole $\omega$ \cite{Axler}. We then have that $|\mathbb{S}^{n-1}|^{-1}Z_j(\xi,x)$ is the reproducing kernel of $\mathcal{H}_{j}(\mathbb{S}^{n-1})$, namely,
\begin{equation}
\label{repKerneleq}
Y_{j}(\omega)=\frac{1}{|\mathbb{S}^{n-1}|}\int_{\mathbb S^{n-1}} Y_{j}(\xi) Z_j(\omega,\xi)d\xi, \quad \mbox{for every } Y_{j}\in\mathcal{H}_{j}(\mathbb{S}^{n-1}). 
\end{equation}

\subsection{Homogeneous extensions and differential operators on $\mathbb{S}^{n-1}$} We write $\mathcal{E}(\Omega)=C^{\infty}(\Omega)$, where $\Omega$ is an open subset of $\mathbb{R}^{n}$ or $\mathbb{S}^{n-1}$. Given a  function $\varphi$ on $\mathbb{S}^{n-1}$, its homogeneous extension (of order 0)  is the function $\varphi^{\upharpoonright}$ defined as $\varphi^{\upharpoonright}(x)=\varphi (x/|x|)$ on $\mathbb{R}^{n}\setminus\{0\}$. It is easy to see that $\varphi\in\mathcal{E}(\mathbb{S}^{n-1})$ if and only if $\varphi^{\upharpoonright}\in\mathcal{E}(\mathbb{R}^{n}\setminus\{0\})$. Furthermore, we define the differential operators $\partial^{\alpha}_{\mathbb{S}^{n-1}}: \mathcal{E}(\mathbb{S}^{n-1}) \to \mathcal{E}(\mathbb{S}^{n-1})$ via 
$$
(\partial^{\alpha}_{\mathbb{S}^{n-1}} \varphi) (\omega)= (\partial^{\alpha} \varphi^{\upharpoonright}) (\omega), \quad \omega\in\mathbb{S}^{n-1}.
$$
We can then consider $L(\partial_{\mathbb{S}^{n-1}})$ for any differential operator $L(\partial)$ defined on $\mathbb{R}^{n}\setminus\{0\}$. In particular, $\Delta_{\mathbb{S}^{n-1}}$ stands for the Laplace-Beltrami operator of the sphere. 

Finally, if $F$ is a function on $\mathbb{R}^{n}$, we simply write $\|F\|_{L^{q}(\mathbb{S}^{n-1})}$ for the $L^{q}(\mathbb{S}^{n-1})$-norm of its restriction to $\mathbb{S}^{n-1}$.

\section{Estimates for partial derivatives of  spherical harmonics}
\label{section: estimates spherical harmonics}
Calder\'{o}n and Zygmund showed \cite[Eq. (4), p. 904]{Calderon} the following estimates for the partial derivatives of a spherical harmonic $Y_{j}\in \mathcal{H}_{j}(\mathbb{S}^{n-1})$, 
\begin{equation}
\label{CZestimate}
\|\partial^{\alpha}_{\mathbb{S}^{n-1}}Y_{j}\|_{L^{\infty}(\mathbb{S}^{n-1})}\leq C_{\alpha,n}\: j^{|\alpha|} \|Y_{j}\|_{L^{\infty}(\mathbb{S}^{n-1})},
\end{equation}
where the constants $C_{\alpha,n}$ depend on the order of differentiation and the dimension on an unspecified way. The same topic is treated in Seeley's article \cite{Se}. 

The goal of this section is to refine \eqref{CZestimate} by exhibiting explicit constants $C_{\alpha,n}$. We also give explicit bounds for the partial derivatives of spherical harmonics in spherical coordinates. Such estimates in spherical coordinates play an important role in the next section. We consider here  $\mathfrak{p}(\theta)=(\mathfrak{p}_1(\theta),\dots,\mathfrak{p}_n(\theta))$,
$$
\mathfrak{p}(\theta)=(\cos\theta_1,\sin\theta_1\cos\theta_2,\dots, \sin\theta_1\cdots\sin\theta_{n-2}\cos\theta_{n-1},\sin\theta_1\dots\sin\theta_{n-2}\sin\theta_{n-1}),
$$ 
where $\theta\in\mathbb{R}^{n-1}$. Naturally, the estimate \eqref{inequality2} below also holds if we choose the north pole to be located at a point other than $(1,0,\dots, 0)$.

\begin{theorem}\label{thestimate} We have the bounds:
\begin{itemize}
\item [(a)] For every solid harmonic $Q_j\in\mathcal H_j(\mathbb R^n)$ and all $\alpha\neq 0$, 
\begin{equation}
\label{inequality1}
\|\partial^{\alpha}Q_j\|_{L^{\infty}(\mathbb S^{n-1})}\leq e^{\frac{n}{4}-\frac{1}{2}}\sqrt{n} \: 2^{\frac{|\alpha|}{2}} j^{|\alpha|+\frac{n}{2}-1}\|Q_j\|_{L^{\infty}(\mathbb S^{n-1})}.   \end{equation}
\item [(b)] For all spherical harmonic $Y_j\in\mathcal H_j(\mathbb {S}^{n-1})$ and all $\alpha\neq 0$,
\begin{equation}\label{inequality2}
\|\partial^{\alpha}_{\theta}(Y_j\circ \mathfrak{p}) \|_{L^{\infty}(\mathbb{R}^{n-1})}\leq e^{\frac{n}{4}-\frac{1}{2}}\sqrt{n} \left( (n+1)^{|\alpha|}-1 \right)2^{\frac{|\alpha|}{2}}j^{|\alpha|+\frac{n}{2}-1} \|Y_j\|_{L^{\infty}(\mathbb S^{n-1})}.
\end{equation}
\item [(c)] For all spherical harmonic $Y_j\in\mathcal H_j(\mathbb {S}^{n-1})$, all $\alpha\neq 0$, and any $\varepsilon>0$,
\begin{equation}\label{inequality3}
\|\partial^{\alpha}_{\mathbb S^{n-1}}Y_j\|_{L^{\infty}(\mathbb{S}^{n-1})}\leq e^{n\left(\frac{1}{4}+\sqrt{2}+3\sqrt{2+4/\varepsilon}\right) -\frac{1}{2} }n^{\frac{|\alpha|+1}{2}} (2+\varepsilon)^{|\alpha|}j^{|\alpha|+\frac{n}{2}-1}|\alpha| !  \|Y_j\|_{L^{\infty}(\mathbb S^{n-1})}.
\end{equation}
\end{itemize}
\end{theorem}

\begin{proof} (a) For (\ref{inequality1}), we assume that $|\alpha|\leq j$, otherwise the result trivially holds.
Our starting point is the same as the inductive step in the proof of \cite[Thm. 4, p. 120]{Se}, namely, the inequality
\begin{equation}\label{step}\int_{\mathbb S^{n-1}}|\partial^{\alpha}Q_j(\omega)|^2d\omega\leq (j-|\alpha|+1)(n+2j-2|\alpha|) \int_{\mathbb S^{n-1}}|\partial^{\beta} Q_j(\omega)|^2d\omega, \end{equation}
valid for all multi-index $\beta$ with $|\beta|=|\alpha|-1$ and $\beta\leq\alpha$. Successive application of (\ref{step}) leads to 
$$
\int_{\mathbb S^{n-1}}|\partial^{\alpha}Q_j(\omega)|^2d\omega\leq \prod_{i=0}^{|\alpha|-1}(j-i)\cdot \prod_{i=1}^{|\alpha|}(n+2j-2i)\int_{\mathbb S^{n-1}} |Q_j(\omega)|^2d\omega.
$$
The coefficient in this bound can be estimated as follows,
\begin{align*}
\prod_{i=0}^{|\alpha|-1}(j-i)\cdot \prod_{i=1}^{|\alpha|}(n+2j-2i)&
\leq j^{2|\alpha|}2^{|\alpha|} \prod_{i=1}^{|\alpha|}\left(1+\frac{n/2-i}{j}\right)
\\
&
\leq 2^{|\alpha|} j^{2|\alpha|} \left(1+\frac{n/2-1}{j}\right)^{|\alpha|}
\\&
\leq 2^{|\alpha|} j^{2|\alpha|} e^{\frac{n}{2}-1}.
\end{align*}
Now, $\partial^{\alpha}Q_{j}\in \mathcal{H}_{j-|\alpha|}(\mathbb{R}^{n})$ and $\|Z_{j-|\alpha|}(\omega,\: \cdot\:)\|_{L^{2}(\mathbb{S}^{n-1})}^{2}=d_{j-|\alpha|}|\mathbb S^{n-1}|$ for each $\omega\in\mathbb{S}^{n-1}$ (cf. \cite[pp. 79--80]{Axler}). Thus, we obtain (cf. \eqref{repKerneleq}), for all $\omega\in\mathbb{S}^{n-1}$, 
\begin{align*}
|\partial^{\alpha} Q_j(\omega)|& \leq \frac{1}{|\mathbb{S}^{n-1}|}\|\partial^{\alpha}Q_j\|_{L^{2}(\mathbb{S}^{n-1})}\|Z_{j-|\alpha|}(\omega,\: \cdot\:)\|_{L^{2}(\mathbb{S}^{n-1})} =\sqrt{ \frac{d_{j-|\alpha|}}{|\mathbb S^{n-1}|}}\: \|\partial^{\alpha}Q_j\|_{L^{2}(\mathbb{S}^{n-1})}
\\
&
\leq e^{\frac{n}{4}-\frac{1}{2}}\sqrt{\frac{n}{ |\mathbb{S}^{n-1}|}}\:  2^{\frac{|\alpha|}{2}} j^{|\alpha|+\frac{n}{2}-1}\| Q_j\|_{L^{2}(\mathbb{S}^{n-1})},
\end{align*}
where we have used  $d_{j-|\alpha|}\leq n j^{n-2}$ (see \eqref{dj}). This shows \eqref{inequality1}.

(b) Our proof of (\ref{inequality2}) is based on the multivariate Fa\`{a} di Bruno formula for the partial derivatives of the composition of functions. Let $m=|\alpha|$. Specializing \cite[Eq. (2.4)]{Faa} to $h=f \circ \mathfrak{p} $, where $f$ is a function on $\mathbb{R}^{n}$, we obtain
$$
\partial^{\alpha}_{\theta} h=\sum_{1\leq|\lambda|\leq m}(\partial^{\lambda}_{x} f)\circ \mathfrak{p}\sum_{(k,l)\in p(\alpha,\lambda)}\alpha!\prod_{j=1}^n \frac{[\partial^{l_j}_{\theta}\mathfrak{p}_{j}]^{k_j}}{(k_j!)[l_j!]^{|k_j|}},
$$
where the set of multi-indices $p(\alpha,\lambda)\subset \mathbb{N}^{2n}$ is as described in \cite[p. 506]{Faa}. We also employ the identity \cite[Cor. 2.9]{Faa}
$$ \alpha!\sum_{|\lambda|=k}\sum_{p(\alpha,\lambda)}\prod_{j=1}^n \frac{1}{(k_j!)[l_j!]^{|k_j|}}=n^k S(m,k),$$
where $S(m,k)$ are the Stirling numbers of the second kind. For such numbers \cite[Thm. 3]{RD1969} we have the estimates\footnote{Actually, $\displaystyle S(m,k)\leq \frac{1}{2}{m\choose k} k^{m-k}$ holds for $1\leq k\leq m-1$ if $m\geq 2$, and $S(m,m)=1$.}
$$
S(m,k)\leq{m\choose k} k^{m-k}, \quad 1\leq k\leq m.
$$
Since obviously $|\partial^{l_{j}}\mathfrak{p}_{j}(\theta)|\leq 1$, we obtain
\begin{equation}\label{sphere}
\|\partial^{\alpha}_{\theta}(f\circ\mathfrak{p})\|_{L^{\infty}(\Omega)}\leq\sum_{k=1}^m  {m\choose k} k^{m-k} n^k  \max_{|\lambda|=k} \|\partial^{\lambda}_{x}f\|_{L^{\infty}(\mathfrak{p}(\Omega))}, 
\end{equation}
for any $\Omega\subseteq \mathbb{R}^{n-1}$ and the corresponding set $\mathfrak{p}(\Omega)\subseteq \mathbb{S}^{n-1}$.
We now apply this inequality to estimate $\partial_{\theta}^{\alpha}(Y_{j}\circ \mathfrak{p})$. Let $Q_{j}\in\mathcal{H}_{j}(\mathbb{R}^{n})$ be the solid harmonic corresponding to $Y_{j}$, clearly $Q_{j}\circ \mathfrak{p}= Y_{j}\circ \mathfrak{p}$ and $\|Q_{j}\|_{L^{\infty}(\mathbb{S}^{n-1})}=\|Y_{j}\|_{L^{\infty}(\mathbb{S}^{n-1})}$. Using \eqref{sphere} with $f=Q_{j}$, the bound \eqref{inequality1}, and the fact that $\partial^{\lambda}_{x} Q_j=0$ if $|\lambda|>j $, we conclude that
\begin{align*}
\|\partial^{\alpha}_{\theta}(Y_j\circ \mathfrak{p})\|_{L^{\infty}(\mathbb R^{n-1})}
&
\leq e^{\frac{n}{4}-\frac{1}{2}}\sqrt{n}2^{\frac{m}{2}}  j^{\frac{n}{2}-1}\|Y_j\|_{L^{\infty}(\mathbb S^{n-1})} \sum_{k=1}^m  {m\choose k} n^k  j^{m-k}    j^{k}
\\
&
= e^{\frac{n}{4}-\frac{1}{2}}\sqrt{n}\left( (n+1)^{m}-1\right) 2^{\frac{m}{2}}j^{m+\frac{n-2}{2}} \|Y_j\|_{L^{\infty}(\mathbb S^{n-1})}.
\end{align*}

(c)  We need to estimate the partial derivatives of $Y^{\upharpoonright}_j=Q_{j}\circ F$, where $Q_{j}\in\mathcal{H}_{j}(\mathbb{R}^{n})$ and $F(x)=x/|x|,$ $x\in\mathbb{R}^{n}\setminus\{0\}$. Instead of using the Fa\`a di Bruno formula to handle directly the partial derivatives of this composition, we will adapt H\"{o}rmander's proof of \cite[Prop. 8.4.1, p. 281]{Hormander} to our problem. Let $0<r<1/2$ and $\omega\in\mathbb{S}^{n-1}$. Note that if $|z-\omega|\leq r$ and we write $z=x+iy$, then
$$
\Re e\:(z_{1}^{2}+\dots +z_{n}^{2})=|x|^{2}-|y|^{2}\geq 1-2r>0.
$$
So, $F$ is holomorphic on this region of $\mathbb{C}^{n}$.
 For $m\geq 1$, we define the sequence of functions
$$G_m(z)=\sum_{|\beta|\leq m}\left( \partial^{\beta} Q_j\right)\left(F(\omega)\right)\frac{\left(F(z)-F(\omega)\right)^{\beta}}{\beta!}.$$
Each $G_m$  is holomorphic when $|z-\omega|\leq r $  and the derivatives of $G_m$ of order $m$ at $z=\omega$ are the same as those of $Y^{\upharpoonright}_j(x)$ at $x=\omega$. We keep $|z-\omega|\leq r$. We have the bound
$$
|F(z)-F(\omega)|\leq 1+ \frac{|z|}{\sqrt{\Re e\:(z_{1}^{2}+\dots +z_{n}^{2})}}< 1+\frac{3}{2\sqrt{1-2r}}=C_{r},
$$
and hence, by \eqref{inequality1}, 
\begin{align*}
|G_m(z)|&
\leq e^{\frac{n}{4}-\frac{1}{2}}\sqrt{n}\:\|Y_j\|_{L^{\infty}(\mathbb S^{n-1})}\sum_{|\beta|\leq \min\{m,j\}} j^{|\beta|+\frac{n}{2}-1}\frac{(C_{r}\sqrt{2})^{|\beta|}}{\beta!}
\\
&
\leq e^{\frac{n}{4}-\frac{1}{2}}\sqrt{n}\:  j^{m+\frac{n}{2}-1}\|Y_j\|_{L^{\infty}(\mathbb S^{n-1})}\sum_{\beta\in\mathbb{N}^{n}} \frac{(C_{r}\sqrt{2})^{|\beta|}}{\beta!}
\\
&=
e^{n\left(\frac{1}{4} +\sqrt{2}C_{r}\right) -\frac{1}{2}}\sqrt{n}\:  j^{m+\frac{n}{2}-1}\|Y_j\|_{L^{\infty}(\mathbb S^{n-1})}.
\end{align*}
The Cauchy inequality applied in the polydisc $|z_{j}-\omega_{j}|\leq r/\sqrt{n}$ yields
$$
|\partial^{\alpha}_{\mathbb S^{n-1}}Y_j(\omega)|=|\partial^{\alpha} G_{|\alpha|}( \omega)|\leq  e^{n\left(\frac{1}{4} +\sqrt{2}C_{r}\right) -\frac{1}{2}} n^{\frac{|\alpha|+1}{2}}r^{-|\alpha|}j^{|\alpha|+\frac{n}{2}-1}\alpha ! \|Y_j\|_{L^{\infty}(\mathbb S^{n-1})}. 
$$
One obtains (\ref{inequality3}) upon setting $r=1/(2+\varepsilon)$.
\end{proof}

\section{Spherical harmonic characterization of ultradifferentiable functions and ultradistributions}
\label{section spherical harmonics}

In this section we characterize ultradifferentiability properties of a function on the sphere in terms of its spherical harmonic expansion. We also obtain a spherical harmonic characterization of ultradistributions on the sphere.

We start by introducing ultradifferentiable functions on $\mathbb{S}^{n-1}$. A weight sequence is simply a positive sequence $(M_p)_{p \in \mathbb{N}}$ of real numbers with $M_{0}=1$. Throughout the rest of the article, we always impose the following assumptions on weight sequences,
 \begin{itemize}
\item [$(M.0)\:$] $p!\leq A_{0}H^{p}_{0} M_{p}$, $p\in\mathbb{N}$, for some $A_{0}>0$ and $H_{0}>1$,
\item [$(M.1)\:$] $M^{2}_{p}\leq M_{p-1}M_{p+1},$  $p\geq 1$,
\item [$(M.2)'$]$M_{p+1}\leq A H^p M_p$, $p\in\mathbb{N}$, for some $A,H>1$.
\end{itemize}

The meaning of these standard conditions is explained in \cite{Komatsu}. The associated function of the sequence $M_p$ is defined as
$$
M(t)=\sup_{p\in\mathbb{N}}\log\frac{t^p}{M_p},\quad t> 0,
$$
and $M(0)=0$. See also \cite{Komatsu} for its properties and the translation of $(M.1)$ and $(M.2)'$ into properties of $M$. In particular, we shall often make use of the inequality 
\begin{equation}
\label{assEstM}
t^{\eta} e^{-M(H^{\eta}t)} \leq A^{\eta}  e^{-M(t)}, \quad \mbox{for all } t>0,
\end{equation}
for any $\eta>0$, implied by $(M.1)$ and $(M.2)'$ \cite[Eq. (3.13), p. 50]{Komatsu}. We also point out that, under $(M.1)$, the condition $(M.0)$ becomes equivalent to the bound $M(t)=O(t)$ \cite[Lemma 3.8]{Komatsu}. As a typical example, we mention $M_{p}=(p!)^{s}$ with $s\geq1$, whose associated function has growth order $M(t)\asymp t^{1/s}$.

We define the space $\mathcal{E}^{\{M_{p}\}}(\mathbb{S}^{n-1})$ of ultradifferentiable functions of Roumieu type (or class $\{M_{p}\}$) as the space of all smooth functions $\varphi\in \mathcal{E}(\mathbb{S}^{n-1})$ such that 
\begin{equation}
\label{ultraeq1}
\sup_{\alpha\in\mathbb{N}}\frac{h^{|\alpha|}\|\partial^{\alpha}_{\mathbb{S}^{n-1}}\varphi\|_{L^{\infty}(\mathbb{S}^{n-1})}}{M_{|\alpha|}}<\infty,
\end{equation}
for some $h>0$. Note that if $M_{p}=(p!)^{s}$ with $s\geq1$, one recovers the spaces of Gevrey differentiable functions on the sphere. In the special but very important case $M_{p}=p!$, we also write  $\mathcal{A}(\mathbb{S}^{n-1})=\mathcal{E}^{\{p!\}}(\mathbb{S}^{n-1})$; this is in fact the space of real analytic functions on $\mathbb{S}^{n-1}$ \cite{Morimoto1998}.

The space $\mathcal{E}^{(M_{p})}(\mathbb{S}^{n-1})$ of ultradifferentiable functions of Beurling type (class $(M_{p})$) is defined by requiring that (\ref{ultraeq1}) holds for every $h>0$. Whenever we consider the Beurling case, we suppose that $M_p$ satisfies the ensuing stronger assumption than $(M.0)$, 
\begin{itemize}
\item [$(NA)$] For each $L>0$ there is $A_{L}>0$ such that  $p!\leq A_{L}L^{p}M_{p}$, $p\in\mathbb{N}$.
\end{itemize}
Notice $(M.1)$ implies that $(NA)$ is equivalent to $M(t)=o(t)$ as $t\to\infty$ \cite[Lemma 3.10]{Komatsu}.

As customary, we  write $\ast=\{M_p\}$ or $(M_p)$ when considering both cases simultaneously. 
It should be noticed that the condition $(M.0)$ (the condition $(NA)$) implies that $\mathcal{A}(\mathbb{S}^{n-1})$ is the smallest among all spaces of ultradifferentiable functions that we consider here, that is, one always has the inclusion $\mathcal{A}(\mathbb{S}^{n-1})\subseteq \mathcal{E}^{\ast}(\mathbb{S}^{n-1})$.

A word about the definition of $\mathcal{E}^{\ast}(\mathbb{S}^{n-1})$ that we have adopted here. Since we have used the differential operators $\partial^{\alpha}_{\mathbb{S}^{n-1}}$ in \eqref{ultraeq1}, $\ast$-ultradifferentiability of $\varphi$ on $\mathbb{S}^{n-1}$ is the same as  $\ast$-ultradifferentiability of its homogeneous extension (of order 0) on  $\mathbb{R}^{n}\setminus\{0\}$, namely,
$$
\varphi\in \mathcal{E}^{\ast}(\mathbb{S}^{n-1}) \quad \mbox{if and only if} \quad \varphi^{\upharpoonright}\in \mathcal{E}^{\ast}(\mathbb{R}^{n}\setminus\{0\}),
$$
with the spaces of ultradifferentiable functions on an open subset of $\mathbb{R}^{n}$ defined in the usual way \cite{Komatsu}. Moreover, in view of the analyticity of the mapping $x\to x/|x|$ and the fact that the pullbacks by analytic functions induce mappings between spaces of $\ast$-ultradifferentiable functions under the assumptions $(M.0)$ ($(NA)$ in the Beurling case), $(M.1)$ and $(M.2)'$ (cf. \cite[Prop. 8.4.1]{Hormander}, \cite[p. 626]{Komatsu2}, \cite{Roumieu1963}), our definition of $\mathcal{E}^{\ast}(\mathbb{S}^{n-1})$ coincides with that of $\ast$-ultradifferentiable functions on compact analytic manifolds via local analytic coordinates.

We are ready to characterize $\mathcal{E}^{\ast}(\mathbb{S}^{n-1})$ in terms the norm decay of projections onto the spaces of spherical harmonics. Recall our convention is to write $\varphi_{j}$ for the projection of $\varphi$ onto $\mathcal{H}_{j}(\mathbb{S}^{n-1})$.

\begin{theorem}\label{t1}  
Let  $\varphi\in L^{2}(\mathbb S^{n-1})$ and let $1\leq q\leq \infty$. The following statements  are equivalent: 
\begin{itemize}
\item [(i)] $\varphi$ belongs to  ${\mathcal E}^{\{M_p\}}(\mathbb S^{n-1})$ (to ${\mathcal E}^{(M_p)}(\mathbb S^{n-1})$) . 

\item [(ii)] $\Delta_{\mathbb S^{n-1}}^p\varphi\in L^2(\mathbb{S}^{n-1})$  for all $p\in\mathbb N$ and there are $h,C>0$ (for every $h>0$ there is $C=C_h>0$) such that 
\begin{equation}\label{Laplacebeltrami}
\|\Delta_{\mathbb S^{n-1}}^p \varphi\|_{L^2(\mathbb S^{n-1})}\leq C h^{-2p} M_{2p}.
\end{equation} 
\item [(iii)] There are $C,h>0$ (for every $h>0$ there is $C=C_h>0$)  such that 
\begin{equation}
\label{t1eq2}
\|\varphi_j\|_{L^{q}(\mathbb S^{n-1})}\leq C e^{-M(hj)}.\end{equation}  
 \end{itemize}
\end{theorem}

\begin{proof}  (i)$\Rightarrow$(ii). The proof of this implication is simple. Indeed, suppose that 
$$|\partial^{\alpha} \varphi^{\upharpoonright}(x)|\leq C h^{-|\alpha|} M_{|\alpha|}, \quad \mbox{for all } 
x\in\mathbb{R}^{n}\setminus\{0\}.
$$
Since
 
$$
\left(\frac{\partial^2}{\partial x_1^2}+\dots+\frac{\partial^2}{\partial x_n^2}\right)^p=\sum_{\alpha_1+\cdots+ \alpha_n=p}\frac{p!}{\alpha_1!\alpha_2!\dots \alpha_n!}\frac{\partial^{2\alpha_1}}{\partial x_{1}^{2\alpha_1}}\cdots\frac{\partial^{2\alpha_n}}{\partial x_n^{2\alpha_n}}
$$
and $$
\sum_{\alpha_1+\dots +\alpha_n=p}\frac{p!}{\alpha_1!\alpha_2!\dots \alpha_n!}=n^p,
$$
the condition $(M.1)$ gives 
$$
\|\Delta_{\mathbb S^{n-1}}^p \varphi\|_{L^2(\mathbb S^{n-1})}\leq \frac{\|\Delta_{\mathbb S^{n-1}}^p \varphi\|_{L^\infty(\mathbb S^{n-1})}}{|\mathbb{S}^{n-1}|^{\frac{1}{2}}}\leq \frac{C}{{|\mathbb{S}^{n-1}|^{\frac{1}{2}}}} (h/\sqrt{n})^{-2p} M_{p}.
$$ 

(ii)$\Rightarrow$(iii). Suppose (\ref{Laplacebeltrami}) holds. The projection of $\varphi$ onto $\mathcal{H}_{j}(\mathbb{S}^{n-1})$ is 
\begin{equation}
\label{eqprojproof}
\varphi_j(\omega)=\frac{1}{|\mathbb S^{n-1}|}\int_{\mathbb S^{n-1}}\varphi(\xi)Z_j(\omega,\xi)d\xi.
\end{equation}
We first assume that $j\geq 1$. The Laplace-Beltrami operator is self-adjoint \cite[Lemma 1]{Se} and each spherical harmonic of degree $j$, such as $Z_{j}(\omega,\xi)$, is an eigenfunction of $\Delta_{\mathbb{S}^{n-1}}$ with eigenvalue $-j(j+n-2)$. Also, $\|Z_{j}(\omega,\: \cdot\:)\|_{L^{2}(\mathbb{S}^{n-1})}=\sqrt{d_{j}|\mathbb S^{n-1}|}\leq j^{\frac{n}{2}-1}\sqrt{n|\mathbb S^{n-1}| }$ (see \eqref{dj} and \cite[pp. 79--80]{Axler}); therefore,
\begin{align*}
|\varphi_j(\omega)|
&
= \frac{1}{j^{p}(j+n-2)^{p}|\mathbb S^{n-1}|}\left|\int_{\mathbb S^{n-1}}(\Delta^{p} \varphi)(\xi) Z_j(\omega,\xi)d\xi\right|
\\
&
\leq \frac{C\sqrt{n}}{|\mathbb{S}^{n-1}|^{\frac{1}{2}}}j^{-2p+\frac{n}{2}-1} h^{-2p}M_{2p}.
\end{align*}
Taking supremum over $\omega$ and infimum over $p$, we conclude that
$$
\|\varphi_{j}\|_{L^{q}(\mathbb{S}^{n-1})}\leq |\mathbb{S}^{n-1}|^{\frac{1}{q}}\|\varphi_{j}\|_{L^{\infty} (\mathbb{S}^{n-1})}\leq C |\mathbb{S}^{n-1}|^{\frac{1}{q}-\frac{1}{2}}\sqrt{n}\:j^{\frac{n}{2}-1} e^{-M(hj)}.
$$
Taking $\eta=n/2-1$ in \eqref{assEstM}, we obtain
$$
\|\varphi_{j}\|_{L^{2}(\mathbb{S}^{n-1})}\leq C|\mathbb{S}^{n-1}|^{\frac{1}{q}-\frac{1}{2}}\sqrt{n}(A/h)^{\frac{n}{2}-1}H^{(\frac{n}{2}-1)^{2}} e^{-M\left(jhH^{1-n/2}\right)}, \quad j\geq 1.
$$
For $j=0$, using (\ref{eqprojproof}), we have $\|\varphi_{j}\|_{L^{q}(\mathbb{S}^{n-1})}\leq C|\mathbb{S}^{n-1}|^{\frac{1}{q}-\frac{1}{2}}$, thus
$$
\|\varphi_{j}\|_{L^{q}(\mathbb{S}^{n-1})}\leq C_{h}C e^{-M\left(jhH^{1-n/2}\right)}, \quad j\geq 0.
$$
with $C_{h}=|\mathbb{S}^{n-1}|^{\frac{1}{q}-\frac{1}{2}}\max\{1,\sqrt{n}(A/h)^{\frac{n}{2}-1} H^{({\frac{n}{2}-1})^{2}}\}$.

(iii)$\Rightarrow$(i). Assume now (\ref{t1eq2}). In view of \eqref{eqprojproof} and \eqref{assEstM}, we may also assume that $q=\infty$. We estimate the partial derivatives of $\varphi$ in spherical coordinates.   Write $\tilde{\varphi}=\varphi \circ \mathfrak{p}$ and $\tilde{\varphi}_{j}=\varphi_{j} \circ \mathfrak{p}$. Let $\alpha\neq 0$. Let $r$ be an integer larger than $n/2+1$. If we combine the estimate \eqref{inequality2} with \eqref{t1eq2}, we obtain
\begin{align*}
\|\partial^{\alpha}_{\theta}\tilde{\varphi}_j\|_{L^{\infty}(\mathbb{R}^{n-1})}
&
\leq e^{\frac{n}{4}-\frac{1}{2}}\sqrt{n}\left(\sqrt{2} (n+1)\right)^{|\alpha|} j^{|\alpha|+\frac{n}{2}-1} \|\varphi_j\|
_{L^{\infty}(\mathbb S^{n-1})}
\\
&
\leq
C\frac{ h^{-r}}{j^{2}} e^{\frac{n}{4}-\frac{1}{2}}\sqrt{n}\left(\sqrt{2} (n+1)/h\right)^{|\alpha|} M_{|\alpha|+r} , \quad j\geq 1.
\end{align*}
Calling $C_{h}= e^{\frac{n}{4}-\frac{1}{2}}h^{-r} \sqrt{n} \pi^{2}/6$, 
we conclude that 
$$
\|\partial^{\alpha}_{\theta}\tilde{\varphi}\|_{L^{\infty}(\mathbb{R}^{n-1})}\leq \sum_{j=1}^{\infty} \|\partial^{\alpha}_{\theta}\tilde{\varphi}_j\|_{L^{\infty}(\mathbb{R}^{n-1})} \leq C_{h} \left(\sqrt{2} (n+1)/h\right)^{|\alpha|} M_{|\alpha|+r}.
$$
The assumption $(M.2)'$ implies  $M_{p+r}\leq  H^{rp}A^{r} H^{\frac{r(r-1)}{2}}M_p$, so
\begin{equation}
\label{isp}
\|\partial_{\theta}^{\alpha} \tilde{\varphi}\|_{L^{\infty}(\mathbb S^{n-1})} \leq C C_{h} A^{r} H^{\frac{r(r-1)}{2}} \left(   H^r \sqrt{2} (n+1)  /h\right)^{|\alpha|} M_{|\alpha|}.
\end{equation}
Setting the north pole at different points of the sphere induces an analytic atlas of $\mathbb{S}^{n-1}$ and $x\to x/|x|$ is analytic on $\mathbb{R}^{n}$. As previously mentioned, the conditions $(M.0)$ ($(NA)$ in the Beurling case), $(M.1)$, and $(M.2)'$ ensure that pullbacks by analytic functions preserve $\ast$-ultradifferentiability. So, $\varphi\in \mathcal{E}^{\ast}(\mathbb{S}^{n-1})$. The inequality \eqref{isp} and the proof of \cite[Prop. 8.1.4]{Hormander} give actually a more accurate result: There are constants $C'_{h}$ and $\ell$, depending also on the sequence $M_p$ and the dimension $n$ but not on $\varphi$, such that
$$
\|\partial^{\alpha}_{\mathbb{S}^{n-1}}\varphi \|_{L^{\infty}(\mathbb{S}^{n-1})}\leq C C'_{h}(\ell h)^{-|\alpha|} M_{|\alpha|}.
$$

\end{proof}

The proof of Theorem \ref{t1} actually yields stronger information than what has been stated. The canonical topology of $\mathcal{E}^{\ast}(\mathbb{S}^{n-1})$ is defined as follows. For each $h>0$, consider the Banach space $\mathcal{E}^{\{M_p\},h}(\mathbb{S}^{n-1})$ of all smooth functions $\varphi$ on $\mathbb{S}^{n-1}$ such that the norm
\begin{equation}
\label{sSeq1}
 \| \varphi \|_{h}=
\sup_{\alpha\in\mathbb{N}}\frac{h^{|\alpha|}\|\partial^{\alpha}_{\mathbb{S}^{n-1}}\varphi\|_{L^{\infty}(\mathbb{S}^{n-1})}}{M_{|\alpha|}}
\end{equation}
is finite. As locally convex spaces, we obtain the $(DFS)$-space and $(FS)$-space
\[
\mathcal{E}^{\{M_p\}}(\mathbb{S}^{n-1})= \varinjlim_{h \rightarrow 0^{+}} \mathcal{E}^{\{M_p\},h}(\mathbb{S}^{n-1}) \quad \mbox{and} \quad \mathcal{E}^{(M_p)}(\mathbb{S}^{n-1}) = \varprojlim_{h \rightarrow \infty} \mathcal{E}^{\{M_p\},h}(\mathbb{S}^{n-1}).\]
What we have shown is that the family of norms \eqref{sSeq1} is tamely equivalent to the norms
\begin{equation} 
\label{sSeq2}
\|\varphi \|'_{h}=\sup_{j\in\mathbb{N}}\: e^{M(h j)}
\|\varphi_j\|_{L^{q} (\mathbb S^{n-1})}, \quad h>0 \quad (1\leq q\leq \infty),
\end{equation}
in the sense that there are positive constants $\ell$ and $L$, only depending on the dimension $n$, the parameter $q$, and the weight sequence, such that  one can find $C_{h}>0$ and $c_{h}>0$ with
$$
c_{h}\| \cdot \|'_{\ell h}\leq \| \cdot \|_{ h} \leq C_{h}\| \cdot \|'_{L h}, \quad \mbox{for all }h>0.
$$
 Working with the family of norms \eqref{sSeq2} is more convenient than \eqref{sSeq1} when dealing with assertions about spherical harmonic expansions.
 
 \begin{proposition}
 \label{shbvp1} Let $\varphi\in\mathcal{E}^{\ast}(\mathbb{S}^{n-1})$. Then its spherical harmonic series expansion $\varphi=\sum_{j=0}^{\infty}\varphi_{j}$ converges in (the strong topology of)  $\mathcal{E}^{\ast}(\mathbb{S}^{n-1})$.
 \end{proposition}
\begin{proof} Let $h>0$. Invoking  \eqref{assEstM} with $\eta=1$,
$$
\|\varphi-\sum_{j=0}^{k} \varphi_{j}\|'_{h}= \sup_{j>k} e^{M\left(hj\right)} \|\varphi_{j}
\| \leq \frac{A}{kh} \:  \|\varphi\|'_{Hh}, \quad \mbox{for each } \quad k\geq 1.$$
\end{proof}

If we specialize our results to the space of real analytic functions and use the fact that the associated function of $p!$ is $M(t)\asymp t$, we obtain the following characterization of $\mathcal{A}(\mathbb{S}^{n-1})=\mathcal{E}^{\{p!\}}(\mathbb{S}^{n-1})$.

\begin{corollary}
\label{shbvc2}
A sequence of spherical harmonics with $\varphi_{j}\in\mathcal{H}_{j}(\mathbb{S}^{n-1})$ gives rise to a real analytic function $\varphi=\sum_{j=0}^{\infty}\varphi_{j}$ on $\mathbb{S}^{n-1}$ if and only if 
$$\limsup_{j\to\infty} \left(\|\varphi_{j}\|_{L^{q}(\mathbb{S}^{n-1})}\right)^{\frac{1}{j}}<1.$$
\end{corollary}

Here is another application of the norms \eqref{sSeq2}. The space of ultradistributions  ${\mathcal{E}^{\ast}}'(\mathbb{S}^{n-1})$ (of class $\ast$) on $\mathbb{S}^{n-1}$ is the strong dual of $\mathcal{E}^{\ast}(\mathbb{S}^{n-1})$. When $\ast=\{p!\}$, one obtains the space of analytic functionals $\mathcal{A}'(\mathbb{S}^{n-1})$ \cite{Morimoto1998}. Given $f\in{\mathcal{E}^{\ast}}'(\mathbb{S}^{n-1})$, we can also define its projection onto $\mathcal{H}_{j}(\mathbb{S}^{n-1})$ as 
$$
f_{j}(\omega)= \frac{1}{|\mathbb{S}^{n-1}|} \langle f(\xi), Z_{j}(\omega,\xi)\rangle, 
$$
where the ultradistributional evaluation in the dual pairing is naturally with respect to the variable $\xi$. Note that, clearly,
\begin{equation}
\label{shbveq2}
\langle f_{j}, \varphi \rangle=\int_{\mathbb{S}^{n-1}} f_{j}(\omega)\varphi(\omega)d\omega = \langle f, \varphi_{j} \rangle, \quad \mbox{for each }\varphi\in \mathcal{E}^{\ast}(\mathbb{S}^{n-1}).
\end{equation}

\begin{theorem}
\label{shbvth2}
Every ultradistribution $f\in{\mathcal{E}^{\{M_{p}\}}}'(\mathbb{S}^{n-1})$ ($f\in{\mathcal{E}^{(M_{p})}}'(\mathbb{S}^{n-1})$) has spherical harmonic expansion 
\begin{equation}
\label{shbveq3} f=\sum_{j=0}^{\infty} f_{j},
\end{equation}
where its spherical harmonic projections $f_{j}$ satisfy
 \begin{equation}
\label{shbveq4} \sup_{j\in\mathbb{R}}e^{-M(hj)} \| f_{j}\|_{L^{q}(\mathbb{S}^{n-1})}<\infty \quad \quad (1\leq q\leq \infty),
\end{equation}
for all $h>0$ (for some $h>0$). Conversely, a series \eqref{shbveq3} converges in the strong topology of ${\mathcal{E}^{\{M_{p}\}}}'(\mathbb{S}^{n-1})$ (of ${\mathcal{E}^{(M_{p})}}'(\mathbb{S}^{n-1})$) if the $L^{q}(\mathbb{S}^{n-1})$-norms of $f_{j}$ have the stated growth properties.
\end{theorem}
\begin{proof} Since $\mathcal{E}^{\ast}(\mathbb{S}^{n-1})$ are Montel spaces, the strong convergence of \eqref{shbveq3} follows from its weak convergence, and the latter is a consequence of Proposition \ref{shbvp1} and \eqref{shbveq2}. For the bound \eqref{shbveq4}, the continuity of $f$ implies that for each $h>0$ (for some $h>0$) there is a constant $C_h$ such that
$$
|\langle f,\varphi \rangle|\leq C_h \| \varphi\|'_{h}, \quad \mbox{for all }\varphi \in \mathcal{A}(\mathbb{S}^{n-1}).
$$
We may assume that $j\geq 1$.
Considering the case $q=2$ of (\ref{sSeq2}), taking $\varphi(\xi)=|\mathbb{S}^{n-1}|^{-1} Z_{j}(\omega,\xi)$, and using the inequalities (\ref{dj}) and \eqref{assEstM}, one has
\begin{align*}
\|f_{j}\|_{L^{q}(\mathbb{S}^{n-1})}&\leq |\mathbb{S}^{n-1}|^{\frac{1}{q}} \|f_{j}\|_{L^{\infty}(\mathbb{S}^{n-1})} \leq |\mathbb{S}^{n-1}|^{\frac{1}{q}-\frac{1}{2}} C_{h} \sqrt{n} j^{\frac{n}{2}-1}e^{M(hj)}
\\
&
\leq |\mathbb{S}^{n-1}|^{\frac{1}{q}-\frac{1}{2}} C_{h} (A/h)^{\frac{n}{2}-1} \sqrt{n} e^{M(jhH^{\frac{n}{2}-1})}.
\end{align*}
\end{proof}

For analytic functionals we have,

\begin{corollary}
\label{shbvc3} A sequence $f_{j}\in\mathcal{H}_{j}(\mathbb{S}^{n-1})$ gives rise to an analytic functional  $f=\sum_{j=0}^{\infty} f_{j}$ on $\mathbb{S}^{n-1}$ if and only if 
$$\limsup_{j\to\infty} \left(\|f_{j}\|_{L^{q}(\mathbb{S}^{n-1})}\right)^{\frac{1}{j}}\leq 1.$$
\end{corollary}

We mention that the strong topologies of the $(FS)$-space ${\mathcal {E}^{\{M_{p}\}}}'(\mathbb{S}^{n-1})$ and the $(DFS)$-space ${\mathcal{E}^{(M_{p})}}'(\mathbb{S}^{n-1})$ can also be induced via the family of norms \eqref{shbveq4} as the projective and inductive limits of the Banach spaces of ultradistributions $f=\sum_{j=0}^{\infty}f_{j}$ satisfying \eqref{shbveq4}.

For each $j\in\mathbb{N}$ select an orthonormal basis of real spherical harmonics $\{Y_{k,j}\}_{k=1}^{d_{j}}$ of $\mathcal{H}_{j}(\mathbb{S}^{n-1})$. It is then clear that every ultradistribution $f\in {\mathcal{E}^{\ast}}'(\mathbb{S}^{n-1})$ and every $\varphi\in \mathcal{E}^{\ast}(\mathbb{S}^{n-1})$ can be expanded as

\begin{equation}
\label{shbveq12}
f=\sum_{j=0}^{\infty}\sum_{k=1}^{d_j}c_{k,j}Y_{k,j}
\end{equation}
and 
\begin{equation}
\label{shbveq13}
\varphi(\omega)=\sum_{j=0}^{\infty}\sum_{k=1}^{d_j}a_{k,j}Y_{k,j}(\omega),
\end{equation}
where the coefficients satisfy 
\begin{equation*}
\sup_{k,j}|c_{k,j}| e^{-M\left(hj\right)}<\infty
\end{equation*}
(for each $h>0$ in the Roumieu case and for some $h>0$ in the Beurling case), 
and 
\begin{equation*}
\sup_{k,j}|a_{k,j}| e^{M\left(hj\right)}<\infty
\end{equation*}
(for some $h>0$ or for each $h>0$, respectively).
Conversely, any series \eqref{shbveq12} and  \eqref{shbveq13} converge in ${\mathcal{E}^{\ast}}'(\mathbb{S}^{n-1})$ and $\mathcal{E}^{\ast}(\mathbb{S}^{n-1})$, respectively, if the coefficients have the stated growth properties. We have used here \eqref{t1eq2}, \eqref{shbveq4}, and \eqref{assEstM}.

From here one easily derives that $\mathcal{E}^{\ast}(\mathbb{S}^{n-1})$ (and hence ${\mathcal{E}^{\ast}}'(\mathbb{S}^{n-1})$) is a nuclear space. We also obtain that $\{Y_{k,j}\}$ is an absolute Schauder basis \cite[p. 340]{Schaefer} for both $\mathcal{E}^{\ast}(\mathbb{S}^{n-1})$ and  ${\mathcal{E}^{\ast}}'(\mathbb{S}^{n-1})$. We end this section with a remark concerning Theorem \ref{t1}.

\begin{remark} It is very important to emphasize that Theorem \ref{t1} is no longer true without the assumption $(M.0)$. 

To see that it is imperative to assume $(M.0)$, we give an example in which the implication (ii)$\Rightarrow$(i) fails without it. In fact, let  $M_{p}$ be any weight sequence for which $(M.1)$ and $(M.2)'$ hold but
$
\lim_{p\to\infty} \left(M_{p}/p!\right)^{\frac{1}{p}}=0.
$
(For example, the sequence $M_{p}=p!^{s}$ with $0<s<1$.) We consider $\varphi(\omega)=Y_{1}(\omega_{1},\dots,\omega_{n})=\omega_{1}$. This function is a spherical harmonic of degree 1, and thus it is an eigenfunction for the Laplace-Beltrami operator corresponding to the eigenvalue $-(n-2)$. Thus, $$
\|\Delta^{p}_{\mathbb{S}^{n-1}}\varphi\|_{L^{2}(\mathbb{S}^{n-1})}\leq \frac{ n^{p}}{|\mathbb{S}^{n-1}|^{\frac{1}{2}}}
$$ and in particular \eqref{Laplacebeltrami} is satisfied for $M_{p}$. If there would be an $h>0$ such that \eqref{ultraeq1} holds with  $M_{p}=p!^{s}$, we would have for the function 
$$
f(t)=\frac{1}{\sqrt{t^{2}+1/2}}
$$
that 
$$
\|f^{(p)}\|_{L^{\infty}(\mathbb{R})}= \sqrt{2} \sup_{t\in\mathbb{R}}  |\partial_{x_{2}}^{p}\varphi^{\upharpoonright} (\sqrt{2}/2, t, 0,\dots, 0)| \leq C' h^{-p} M_{p}, \quad \mbox{for all } p\in\mathbb{N},
$$
for some $C'>0$. 
But then $f$ would be analytically continuable to the whole $\mathbb{C}$ as an entire function, which is impossible because $f$ has branch singularities at $t=\pm i\sqrt{2}/2$.

On the other hand, note that in establishing the implications (i)$\Rightarrow$(ii)$\Rightarrow$(iii) the condition $(M.0)$ plays no role because we have only made use there of $(M.1)$ and $(M.2)'$.
\end{remark}

\section{Boundary values of harmonic functions}
\label{section: boundary values harmonic}
We now generalize the results from \cite{Estrada} to ultradistributions. We shall characterize all those harmonic functions on the open unit ball $\mathbb{B}^{n}$ that admit ultradistributional boundary values on $\mathbb{S}^{n-1}$ in terms of their growth near the boundary. Our characterization applies for sequences satisfying the additional conditions discussed below.

Let us fix some notation and terminology. We write $\mathcal{H}(\mathbb{B}^{n})$ for the space of all harmonic functions on $\mathbb{B}^{n}$. We say that $U\in \mathcal{H}(\mathbb{B}^{n})$ has ultradistribution boundary values in the space ${\mathcal{E}^{\ast}}'(\mathbb{S}^{n-1})$ if there is $f\in {\mathcal{E}^{\ast}}'(\mathbb{S}^{n-1})$ such that
\begin{equation}
\label{bvheq1}
\lim_{r\to1^{-}} U(r\omega)= f(\omega) \quad \mbox{in }  {\mathcal{E}^{\ast}}'(\mathbb{S}^{n-1}).
\end{equation}
Since ${\mathcal{E}^{\ast}}'(\mathbb{S}^{n-1})$ is Montel, the converge of \eqref{bvheq1} in the strong topology is equivalent to weak convergence, i.e., 
\begin{equation}
\label{bvheq2}
\lim_{r\to1^{-}} \langle U(r\omega),\varphi(\omega)\rangle= \lim_{r\to1^{-}} \int_{\mathbb{S}^{n-1}} U(r\omega)\varphi(\omega)d\omega= \langle f, \varphi \rangle, 
\end{equation}
for each $\varphi \in \mathcal{E}^{\ast}(\mathbb{S}^{n-1}).$

We first show that \eqref{bvheq1} holds with $U$ being the Poisson transform of $f$. For this, our assumptions are the same as in the previous section, i.e., $(M.1)$, $(M.2)'$ and $(M.0)$ ($(NA)$ in the Beurling case).
The Poisson kernel of $\mathbb{S}^{n-1}$ is \cite{Axler}
\begin{equation}
\label{bvheq3}
P(x,\xi)=\frac{1}{|\mathbb{S}^{n-1}|}\frac{1-|x|^2}{|x-\xi|^n}=\frac{1}{|\mathbb{S}^{n-1}|}\sum _{j=0}^{\infty} |x|^{j} Z_{j}\left(\frac{x}{|x|},\xi\right), \quad \xi\in\mathbb{S}^{n-1},\ x\in\mathbb{B}^{n}.
\end{equation}
Since $P$ is real analytic with respect to $\xi$, we can define the Poisson transform of $f\in {\mathcal{E}^{\ast}}'(\mathbb{S}^{n-1})$ as
\begin{equation}
\label{bvheq4}
P[f](x)=\langle f(\xi), P(x,\xi)\rangle, \quad x\in\mathbb{B}^{n}.
\end{equation}
Clearly, $P[f]\in \mathcal{H}(\mathbb{B}^{n})$ and, by \eqref{bvheq3},  $P[f](r\omega)=\sum_{j=0}^{\infty}r^{j}f_j(\omega)$.

\begin{proposition}
\label{bvhp1} For each $f\in {\mathcal{E}^{\ast}}'(\mathbb{S}^{n-1})$ and $\varphi \in \mathcal{E}^{\ast}(\mathbb{S}^{n-1})$, we have
\begin{equation}
\label{bvheq5}
\lim_{r\to1^{-}} P[f](r\omega)= f(\omega) \quad \mbox{in }  {\mathcal{E}^{\ast}}'(\mathbb{S}^{n-1})
\end{equation}
and
\begin{equation}
\label{bvheq6}
\lim_{r\to1^{-}} P[\varphi](r\omega)= \varphi(\omega) \quad \mbox{in }  \mathcal{E}^{\ast}(\mathbb{S}^{n-1}).
\end{equation}
\end{proposition}
\begin{proof} Due to the Montel property of these spaces (which also implies they are reflexive), it is enough to verify weak convergence of the Poisson transform in both cases in order to prove strong convergence of \eqref{bvheq5} and \eqref{bvheq6}. By Theorem \ref{shbvth2} (or Theorem \ref{t1}), we have that $\langle f,\varphi \rangle=\sum_{j=0}^{\infty} \langle f_{j},\varphi_{j} \rangle$; Abel's limit theorem on power series then yields
$$
\lim_{r\to1^{-}} \int_{\mathbb{S}^{n-1}} P[f](r\omega)\varphi(\omega)d\omega= \lim_{r\to1^{-}} \langle f(\omega),P[\varphi](r\omega)\rangle=  \lim_{r\to1^{-}}\sum_{j=0}^{\infty} r^{j} \langle f_{j},\varphi_{j} \rangle=
 \langle f, \varphi \rangle.
$$
\end{proof}

We now deal with the characterization of harmonic functions $U$ that satisfy (\ref{bvheq1}). This characterization is in terms of the associated function of $M_{p}/p!$, which we denote by $M^{\ast}$ as in \cite{Komatsu}, i.e., the function
$$
M^{\ast}(t)=\sup_{p\in\mathbb{N}} \log \left(\frac{p!t^{p}}{M_p}\right) \quad \mbox{for }t>0
$$
and $M^{\ast}(0)=0.$ We need two extra assumptions on the sequence, namely,
\smallskip

 \begin{itemize}
\item [$(M.1)^{\ast}$] $M_{p}/p!$ satisfies $(M.1)$,
\item [$(M.2)$\ ] $M_{p+q} \leq AH^{p+q}M_qM_q$, $p,q \in \mathbb{N}$, for some $A,H \geq 1$.
\end{itemize}
Naturally, $(M.1)^{\ast}$ implies $(M.0)$ and $(M.1)$ while $(M.2)$ is stronger than $(M.2)'$.

Note that $(M.1)^{\ast}$ delivers essentially two cases. Either $(NA)$ holds or there are constants such that $C_{1}L^{p}_{1}p!\leq M_{p}\leq C_{2}L^{p}_2p!$. In the latter case we may assume that $M_{p}=p!$ as for any such a sequence $\mathcal{E}^{\{M_{p}\}}(\mathbb{S}^{n-1})= \mathcal{A}(\mathbb{S}^{n-1})$. When $(NA)$ holds $M^{\ast}(t)$ is finite for all $t\in[0,\infty)$, whereas $M_{p}=p!$ gives $M^{\ast}(t)=0$ for $0\leq t\leq1$ and $M^{\ast}(t)=\infty$ for $t>1$. In the $(NA)$ case we also have $M^{\ast}(t)=0$ for $t\in[0,M_1]$. The importance of the assumptions $(M.1)^{\ast}$ and $(M.2)$ lies in the ensuing lemma of Petzsche and Vogt:

\begin{lemma}[\cite{Petzsche84}]
\label{bvhl2} Suppose that $M_p$ satisfies $(M.1)^{\ast}$ and $(M.2)$. Then, there are constants $L,\ell>0$ such that
$$
\inf_{y>0} (M^{\ast}(1/y)+ty)\leq M(\ell t)+\log L, \quad \mbox{for all }t>0.
$$

\end{lemma}

We then have,

\begin{theorem} \label{bvht1} Assume $M_p$ satisfies $(M.1)^{\ast}$ and $(M.2)$. Then, a harmonic function $U\in\mathcal{H}(\mathbb{B}^{n})$ admits boundary values in ${\mathcal {E}^{\{M_p\}}}'(\mathbb S^{n-1})$  (in ${\mathcal E^{(M_p)}}'(\mathbb S^{n-1})$)
if and only if for each $h>0$ there is $C=C_h>0$ (there are $h>0$ and $C>0$) such that
\begin{equation}\label{harmonic}|U(x)|\leq C e^{M^{\ast}\left(\frac{h}{1-|x|}\right)} \quad \mbox{for all } x\in\mathbb{B}^{n}.
\end{equation}
In such a case $U=P[f]$, where $f$ is its boundary ultradistribution given by \eqref{bvheq1}.    
\end{theorem} 
\begin{proof}
Suppose $U(x)=P[f](x)$ with  $f\in {\mathcal{E}^{\ast}}'(\mathbb{S}^{n-1})$. Then, 
$$
U(r\omega)=\sum_{j=0}^{\infty} r^j f_j(\omega). $$
If 
$\|f_j\|_{L^{\infty}(\mathbb{S}^{n-1})}\leq C e^{M(h j)},$
for a fixed $h>0$, the inequality \eqref{assEstM} gives 
\begin{align*}
|U(r\omega)|&\leq \sum_{j=0}^{\infty}|f_j(\omega)|r^j
=C+\frac{A^{2}}{h^{2}}\sum_{j=1}^{\infty}\frac{1}{j^{2}}|f_j(\omega)| e^{-M(hj)}e^{M(hH^{2}j)}r^j
\\
&
\leq C\left(1+\frac{A^{2}\pi^{2}}{6h^{2}}\right) \sup_{j\in\mathbb N} r^j e^{M(hH^{2}j)}.
\end{align*}
Now, 
$$
\sup_{j\in\mathbb {N}}r^j e^{M(hH^{2}j)}=\sup_{p\in\mathbb N}  \frac{(H^{2}h)^{p}}{M_p} \sup_{j\in\mathbb {N}} r^j j^p
$$
and 
$$
\sup_{j\in\mathbb{N}}r^j j^p\leq\sum_{j=0}^{\infty} r^j j^p\leq\sum_{j=0}^{\infty}\frac{(j+p)!}{j!}r^j=\left(\frac{1}{1-r}\right)^{(p)}=\frac{p!}{(1-r)^{p+1}}<\frac{(p+1)!}{(1-r)^{p+1}}\:.
$$
Therefore, by $(M.2)'$, 
$$|U(r\omega)|\leq C\frac{A}{H^{3}h}\left(1+\frac{A^{2}\pi^{2}}{6h^{2}}\right)  \sup_{p}\frac {(p+1)! (H^{3}h)^{p+1}}{M_{p+1}(1-r)^{p+1}}\leq CC_{h} e^{M^{\ast}\left(\frac{H^{3}h}{1-r}\right)}.$$

Assume now that \eqref{harmonic} holds for each $h>0$ (for some $h>0$). Every harmonic function on $\mathbb{B}^{n}$ can be written as
$$ 
U(r\omega)=\sum_{j=0}^{\infty} r^{j}f_{j}(\omega),
$$
with each $f_{j}$ a spherical harmonic of degree $j$. By Proposition \ref{bvhp1}, it is enough to check that $f=\sum_{j=0}^{\infty}f_{j}\in{\mathcal{E}^{\ast}}'(\mathbb{S}^{n-1})$, because in this case $U=P[f]$ and $f$ would be the boundary ultradistribution of $U$. By Theorem \ref{shbvth2}, it is then suffices to verify that the sequence $f_{j}$ satisfies the bounds \eqref{shbveq4} for each $h>0$ (for some $h>0$). Here we use $q=\infty$. 
Fix $h>0$ and assume that \eqref{harmonic}  holds. One clearly has 
$$
f_{j}(\omega)= \frac{1}{r^{j}|\mathbb S^{n-1}|}\int_{\mathbb S^{n-1}} U(r\xi) Z_j(\omega,\xi)d\xi.
$$
When $j=0$, we obtain $f_{0}=U(0)$ and so $\|f_{0}\|_{L^{\infty}(\mathbb{S}^{n-1})}\leq C e^{M^{\ast}(h)}$. Keep now $j\geq 1$. Since the zonal harmonic satisfies $\|Z_j(\cdot,\xi)\|_{L^{\infty}(\mathbb S^{n-1})}= d_{j}\leq n j^{n-2}$ \cite[p. 80]{Axler}, we obtain, for all $j\geq 1$,
\begin{equation*}
\label{h1}
\|f_j\|_{L^{\infty}(\mathbb S^{n-1})}\leq C n j^{n-2} \inf_{0<r<1}  r^{-j} e^{M^{\ast}(\frac{h}{1-r})}.
\end{equation*} 
Performing the substitution $r=e^{-y}$, and using Lemma \ref{bvhl2} and $M^{\ast}(t)=0$ for $t\leq M_{1}$, 
\[
\|f_j\|_{L^{\infty}(\mathbb S^{n-1})}\leq  C C_{h} n j^{n-2}\exp\left(\inf_{0<y<\infty} M^{\ast}\left(2h/y\right)+jy\right)\leq C C_{h} L n j^{n-2}e^{M(2\ell h j)}.
\]
Finally, using the estimate (\ref{assEstM}), we conclude that there is $C'_{h}$ such that
$$
\|f_j\|_{L^{\infty}(\mathbb S^{n-1})}\leq C C'_{h} e^{M(2\ell h j)}, \quad \mbox{for each } j\in\mathbb{N}.
$$
\end{proof} 
When $M_{p}=p!$, the bound \eqref{harmonic} holds for any arbitrary harmonic function since $M^{\ast}(t)=\infty$ for $t>1$. Hence,
\begin{corollary}\label{bvhc1} Any harmonic function $U\in\mathcal{H}(\mathbb{B}^{n})$ can be written as the Poisson transform  $U=P[f]$ of  an analytic functional $f$ on $\mathbb{S}^{n-1}$.
\end{corollary} 

Suppose that $M_{p}$ satisfies $(NA)$. Consider the family of Banach spaces
$$
{\mathcal H}^{M_p,h}(\mathbb B^n)=\{U\in{\mathcal H}(\mathbb B^n):\|U\|_{{\mathcal H}^{M_p,h}(\mathbb B^n)}=\sup_{x\in\mathbb{B}^{n}} |U(x)|e^{-M^{\ast}(\frac{h}{1-|x|})}<\infty\}.
$$
We define the Fr\'{e}chet  and $(LB)$ spaces of harmonic functions
$$
\mathcal{H}^{\{M_p\}}(\mathbb{B}^{n}) = \varprojlim_{h \rightarrow 0^{+}} \mathcal{H}^{\{M_p\},h}(\mathbb{B}^{n}) \quad \mbox{ and }\quad \mathcal{H}^{(M_p)}(\mathbb{B}^{n}) = \varinjlim_{h \rightarrow \infty} \mathcal{H}^{\{M_p\},h}(\mathbb{B}^{n}).
$$

This definition still makes sense for $\{p!\}$ because for $h<1$ we have
$$
\sup_{x\in\mathbb{B}^{n}} |U(x)|e^{-M^{\ast}\left(\frac{h}{1-|x|}\right)}=\sup_{|x|\leq 1-h}|U(x)|.$$ 
In this case we obtain the space of all harmonic functions $\mathcal{H}(\mathbb{B}^{n})=\mathcal{H}^{\{p!\}}(\mathbb{B}^{n})$ with the canonical topology of uniform convergence on compact subsets of $\mathbb{B}^{n}$. By Theorem \ref{bvht1}, the mapping $\operatorname*{bv}(U)=f$, with $f$ given by \eqref{bvheq1}, provides a linear isomorphism from $\mathcal{H}^{\ast}(\mathbb{B}^{n})$ onto  ${\mathcal{E}^{\ast}}'(\mathbb{S}^{n-1})$ if $M_{p}$ satisfies $(M.1)^{\ast}$ and $(M.2)$. Our proof given above actually yields a topological result:

\begin{theorem}
\label{bvhth2} Suppose $M_{p}$ satisfies $(M.1)^{\ast}$ and $(M.2)$. The boundary value mapping $$
\operatorname*{bv}: \mathcal{H}^{\ast}(\mathbb{B}^{n})\to {\mathcal{E}^{\ast}}'(\mathbb{S}^{n-1})$$
is a topological vector space isomorphism with the Poisson transform $$P:{\mathcal{E}^{\ast}}'(\mathbb{S}^{n-1})\to \mathcal{H}^{\ast}(\mathbb{B}^{n})$$ as inverse.

\end{theorem}

\begin{remark}
\label{bvhrk2} Suppose $M_p$ satisfies $(M.0)$ ($(NA)$ in the Beurling case). Theorem \ref{bvhth2} is valid if one replaces $(M.1)^{\ast}$ by the condition
\smallskip

\begin{itemize}
\item [$(M.4)$] $M_p\leq L^{p+1}p! M^{\ast}_{p},$ $p\in\mathbb{N},$ for some  $L\geq 1$.
\end{itemize}
Here $M_{p}^{\ast}$ is the convex regularization of $M_{p}/p!$, namely, the sequence
$$
M_{p}^{\ast}= \sup_{t>0}\frac{t^{p}}{e^{M^{\ast}(t)}}.
$$
In fact,  $p!M_{p}^{\ast}$ satisfies $(M.1)^{\ast}$ and, under $(M.4)$, gives rise to the same ultradistribution spaces as $M_p$. We mention that strong non-quasianalyticity (i.e., Komatsu's condition $(M.3)$ \cite{Komatsu}) automatically yields $(M.4)$, as was shown by Petzsche \cite[Prop. 1.1]{Petzsche88}. Furthermore, Petzsche and Vogt \cite[Sect. 5]{Petzsche84} proved under the assumption $(M.2)$ that $(M.4)$ is equivalent to the so-called Rudin condition:
\begin{itemize}
\item [$(M.4)''$] $\displaystyle \max_{q\leq p}\left(\frac{M_{q}}{q!}\right)^{\frac{1}{q}}\leq A\left(\frac{M_{p}}{p!}\right)^{\frac{1}{p}},$ $p\in\mathbb{N}$, for some $A>0$,
\end{itemize}
which is itself equivalent to the property that $\mathcal{E}^{\ast}(\mathbb{S}^{n-1})$ is inverse closed (cf. \cite{Rainer-S,Rudin62}).
\end{remark}

\section{ The support of ultradistributions on the sphere} 
\label{section: support sphere}
This section is devoted to characterizing the support of non-quasianalytic ultradistributions in terms of (uniform) Abel-Poisson summability of their spherical harmonic expansions. Our assumptions on the weight sequence are $(M.1)$, $(M.2)'$, and non-quasianalyticity \cite{Komatsu}, that is,
 \begin{itemize}
\item [$(M.3)'$]  $\displaystyle \sum_{p=1}^{\infty}\frac{M_{p-1}}{M_p}<\infty$. 
\end{itemize}
Note that $(NA)$ is automatically fulfilled because of $(M.3)'$ \cite[Lemma 4.1, p. 56]{Komatsu}.

To emphasize we are assuming $(M.3)'$, we write ${\mathcal{D}^{\ast}}'(\mathbb{S}^{n-1})={\mathcal{E}^{\ast}}'(\mathbb{S}^{n-1})$. By the Denjoy-Carleman theorem \cite{Komatsu}, the support of an ultradistribution $f\in{\mathcal{D}^{\ast}}'(\mathbb{S}^{n-1})$ can be defined in the usual way. Since the natural inclusion $\mathcal{D}'(\mathbb{S}^{n-1})\subset {\mathcal{D}^{\ast}}'(\mathbb{S}^{n-1})$ is support preserving, Theorem \ref{bvhth4} below contains Gonz\'{a}lez Vieli's characterization of the support of Schwartz distributions on the sphere \cite{GonzalezV2016}. The key to the proof of our generalization is the ensuing lemma about the Poisson kernel. Given a non-empty closed set $K\subset \mathbb{S}^{n-1}$ and a weight sequence $N_p$, we consider the family of seminorms
$$
\|\varphi\|_{\mathcal{E}^{\{N_p\},h}(K)}=\sup_{\alpha\in\mathbb{N}}\frac{h^{|\alpha|}\|\partial^{\alpha}_{\mathbb{S}^{n-1}}\varphi\|_{L^{\infty}(K)}}{N_{|\alpha|}}.
$$
\begin{lemma}
\label{bvhl3} Let $K_{1}$ and $K_{2}$ be two disjoint non-empty closed subsets of $\mathbb{S}^{n-1}$. Write $P_{r\omega}(\xi)=P(r\omega,\xi)$, regarded as a function in the variable $\xi\in\mathbb{S}^{n-1}$. Then, there are two positive constants $\ell$ and $C$, only depending on $K_1$ and $K_2$, such that 
$$
\|P_{r\omega}\|_{\mathcal{E}^{\{p!\},\ell}(K_{2})}\leq C(1-r), \quad \mbox{for all } \omega\in K_{1} \mbox{ and } \frac{1}{2}\leq r<1.
$$
\begin{proof} For the sake of convenience, we introduce the spherical type distance 
$$d(\omega,\xi)=1-\omega \cdot \xi .$$
Let $V\subset \mathbb{S}^{n-1}$ be open such that $K_1\cap \overline{V}=\emptyset$ and $K_{2}\subset V$.  Set $\rho=d(K_1, V)$. Note that if $\omega\in K_{1}$ and $\xi\in V$, the term in the denominator of the Poisson kernel,
$$
P(r\omega,\xi)=\frac{1}{|\mathbb{S}^{n-1}|}\frac{1-r^{2}}{(1-2r \omega\cdot \xi+r^2)^{\frac{n}{2}}}\:,
$$
can be estimated by using the lower bound
$$
1-2r \omega\cdot \xi+r^2= (1-r)^{2}+2r(1-\omega\cdot \xi)>2r \rho.
$$ 

We estimate the derivatives of the Poisson kernel in spherical coordinates $\mathfrak{p}(\theta)$ where the north pole is chosen to be located at an arbitrary point of the sphere. Keep $\omega\in K_{1}$ and $1/2\leq r<1$ arbitrary. Let $V'\subset \mathbb{R}^{n-1}$ be such that $V=\mathfrak{p}(V')$. Call $m=|\alpha|$. Using the estimate (\ref{sphere}) and the obvious inequality $m^{m}\leq e^{m-1} m!$ , we obtain
\begin{align*}
\|\partial_{\theta}^{\alpha}(P_{r\omega}\circ \mathfrak{p})\|_{L^{\infty}(V')}
&
\leq \frac{1-r^2}{|\mathbb S^{n-1}|}\sum_{k=1}^m {m \choose k}k^{m-k}n^k\frac{\Gamma\left(\frac{n}{2}+k\right)}{\Gamma\left(\frac{n}{2}\right)} \sup_{\xi \in V}\frac{(2r)^k|\omega|^{k}}{(1-2r \omega\cdot \xi)+r^2)^{k+\frac{n}{2}}}
\\
&
<\frac{1-r^2}{(2r\rho)^{\frac{n}{2}}|\mathbb S^{n-1}|}m^{m}  \sum_{k=1}^m{m\choose k}  \left(1+\frac{\frac{n}{2}-1}{k}\right)^{k}\left(\frac{n}{\rho}\right)^{k}
\\
&
< \frac{3e^{\frac{n}{2}-2}}{2\rho^{\frac{n}{2}}|\mathbb S^{n-1}|}(1-r)m!\left(e\left(\frac{n}{\rho}+1\right)\right)^{m}= C_{1}(1-r)\ell_{1}^{-|\alpha|}|\alpha|!.
\end{align*}
Varying the north poles, we can cover $K_{2}$ by a finite number of open subsets of $V$, each of which parametrized by a system of invertible spherical coordinates. Inverting the polar coordinates on each of the open sets of this covering with the aid of \cite[Prop. 8.1.4]{Hormander}, we deduce that there are $\ell,C>0$, depending only on $V$, such that  
$$
\frac{\ell^{\alpha}\|\partial_{\mathbb{S}^{n-1}}^{\alpha}P_{r\omega}\|_{L^{\infty}(K_{2})}}{|\alpha|!}\leq C(1-r), \quad \mbox{for all } \alpha\in\mathbb{N}^{n}.$$
 This completes the proof of the lemma.
\end{proof}
\end{lemma}

We are ready to state and prove our last result:

\begin{theorem}
\label{bvhth4} 
Let $f=\sum_{j=0}^{\infty}f_{j}\in{\mathcal{D}^{\ast}}'(\mathbb{S}^{n-1})$ and let $\Omega$ be an open subset of $\mathbb{S}^{n-1}$. If
\begin{equation}
\label{support}
 \lim_{r\to 1^{-}}\sum_{j=1}^{\infty}r^j f_j(\omega)=\lim_{r\to1}P[f](r\omega)=0
 \end{equation}
holds uniformly for $\omega$ on compact subsets of $\Omega$, then $\Omega\subseteq \mathbb S^{n-1}\setminus\operatorname*{supp} f$. 

Conversely,  \eqref{support} holds uniformly on any compact subset of 
$\mathbb{S}^{n-1}\setminus\operatorname*{supp} f$.  
\end{theorem}
\begin{proof} The first part follows immediately from Proposition \ref{bvhp1}. Indeed, let $\varphi\in{\mathcal E}^{\ast}(\mathbb S^{n-1})$ be an arbitrary test function such that $\operatorname*{supp}\varphi\subset \Omega$. Then,
$$
\langle f,\varphi \rangle= \lim_{r\to 1^-}\int_{\mathbb S^{n-1}}P[f](r\omega)\varphi(\omega)d\omega
=\lim_{r\to 1^-}\int_{\operatorname*{supp}\varphi}P[f](r\omega)\varphi(\omega)d\omega=0,
$$
which gives that $f$ vanishes on $\Omega$.

Conversely, since we have the dense and continuous embeddings 
$\mathcal{E}^{(M_{p})}(\mathbb{S}^{n-1})\hookrightarrow \mathcal{E}^{\{M_{p}\}}(\mathbb{S}^{n-1})$ (by Proposition \ref{shbvp1} the linear span of the spherical harmonics is dense in both spaces), we have the natural inclusion ${\mathcal{E}^{\{M_{p}\}}}'(\mathbb{S}^{n-1})\rightarrow {\mathcal{E}^{(M_{p})}}'(\mathbb{S}^{n-1})$ which is obviously support preserving. Thus, we may just deal with the case $f\in{\mathcal{E}^{(M_{p})}}'(\mathbb{S}^{n-1})$.  Let $K_1$ be closed such that $K_1\cap \operatorname*{supp} f=\emptyset$. Select a closed subset of the sphere $K_{2}$ such that $K_{1}\cap K_{2}=\emptyset$ and $\operatorname*{supp}f\subset \operatorname*{int}K_{2}$. There are then $C_1$ and $h>0$ such that
$$
|\langle f,\varphi \rangle|\leq C_1\|\varphi\|_{\mathcal{E}^{\{M_{p}\},h}(K_{2})}, \quad \mbox{for all }\varphi\in \mathcal{E}^{(M_{p})}(\mathbb{S}^{n-1}).
$$
The sequence $M_p$ satisfies $(NA)$, hence, given $\ell$, one can find $C_{2}>0$, depending only on $h$ and $\ell$, such that $\|\varphi\|_{\mathcal{E}^{\{M_{p}\},h}(K_{2})}\leq C_{2}\|\varphi\|_{\mathcal{E}^{\{p!\},\ell}(K_{2})}$ for all $\varphi\in \mathcal{A}(\mathbb{S}^{n-1})$. Using this with $\varphi=P_{r\omega}$ and employing Lemma \ref{bvhl3},
$$
|P[f](r\omega)|=|\langle f,P_{r\omega}\rangle|\leq C_{1}C_{2}C(1-r), \quad \mbox{for all }\omega\in K_{1} \mbox{ and } \frac{1}{2}\leq  r<1,
$$ 
whence (\ref{support}) holds uniformly for $\omega\in K_{1}$.
 \end{proof}


\begin{thebibliography}{99}

\bibitem{alvarez2007} J. Alvarez, M. Guzm\'{a}n-Partida, S. P\'{e}rez-Esteva, \emph{Harmonic extensions of distributions,} Math. Nachr. \textbf{280} (2007), 1443--1466. 

\bibitem{Han} K.~Atkinson, W.~Han, \emph{Spherical harmonics and approximations on the unit sphere: An introduction}, Springer, Berlin-Heidelberg, 2012.

\bibitem{Axler} S.~Axler, P.~Bourdon, W.~Ramey, \emph{Harmonic function theory}, Springer-Verlag, New York, 1992. 


\bibitem{Calderon} A.~P.~Calder\'{o}n, A.~Zygmund, \emph{Singular integral operators and differential equations}, Amer. J. Math. \textbf{79} (1957), 901--921.

\bibitem{C-K-P} R.~Carmichael, A.~Kami\'nski, S.~Pilipovi\'c, \emph{Boundary values and convolution in ultradistribution spaces}, World Scientific Publishing Co. Pte. Ltd., Hackensack, New Jersey, 2007.

\bibitem{C-M} 
    R.~Carmichael, D.~Mitrovi\'{c}, 
    \emph{Distributions and analytic functions}, 
    Pitman Research Notes in Mathematics Series, 206, Longman Scientific \& Technical, Harlow; John Wiley \& Sons, Inc., New York, 1989.


 \bibitem{Faa}  G.~M.~Constantine,  T.~H.~Savits, \emph{A multivariate Fa\`{a} di Bruno formula with applications}, Trans. Amer. Math. Soc. \textbf{348} (1996), 503--520.


\bibitem{D-R2014} A.~Dasgupta, M.~Ruzhansky, \emph{Gevrey functions and ultradistributions on compact Lie groups and homogeneous spaces,} Bull. Sci. Math. \textbf{138} (2014), 756--782.

\bibitem{D-R2016} A.~Dasgupta, M.~Ruzhansky, \emph{Eigenfunction expansions of ultradifferentiable functions and ultradistributions,} Trans. Amer. Math. Soc. \textbf{368} (2016), 8481--8498.

\bibitem{D-P-VBV2015} P.~Dimovski, S.~Pilipovi\'{c}, J.~Vindas, \emph{Boundary values of holomorphic functions in translation-invariant distribution spaces}, Complex Var. Elliptic Equ. \textbf{60} (2015), 1169--1189. 


\bibitem{Estrada} R.~ Estrada, R.~P.~Kanwal, \emph{Distributional boundary values of harmonic and analytic functions}, J. Math. Anal. Appl. \textbf {89} (1982), 262--289. 


\bibitem{f-g-g}	
    C.~Fern\'{a}ndez, A.~Galbis, M.~C.~G\'{o}mez-Collado, 
    \emph{(Ultra)distributions of $L_p$-growth as boundary values of holomorphic functions}, Rev. R. Acad. Cienc. Exactas F\'{i}s. Nat. Ser. A Mat. \textbf{97} (2003),  243--255.

\bibitem{GonzalezV2016} F.~J. Gonz\'{a}lez Vieli, \emph{Abel means for orthogonal expansions of distributions on spheres, balls and simplices,} J. Math. Anal. Appl. \textbf{433} (2016), 496--508.

\bibitem{Hormander} L.~H\"ormander, \emph{The analysis of linear partial differential operators. I. Distribution theory and Fourier analysis}, Springer-Verlag, 
Berlin, 1990.

\bibitem{Komatsu} H.~Komatsu, \emph{Ultradistributions I: Structure theorems and a characterization}, J. Fac. Sci. Tokyo Sect. IA Math. \textbf{20} (1973), 25--105.

\bibitem{Komatsu2} H.~Komatsu, \emph{Ultradistributions II: The kernel theorem and ultradistributions with support in a submanifold,}  J. Fac. Sci. Tokyo Sect. IA Math. \textbf{24} (1977), 607--628.
 

\bibitem{Morimoto1998} M.~Morimoto, \emph{Analytic functionals on the sphere,} American Mathematical Society, Providence, RI, 1998.
 
\bibitem{Morimoto2000}M.~Morimoto, M.~Suwa, \emph{A characterization of analytic functionals on the sphere. II,} in: Proceedings of the Second ISAAC Congress, Vol. 2 (Fukuoka, 1999), pp. 799--807, Int. Soc. Anal. Appl. Comput., 8, Kluwer Acad. Publ., Dordrecht, 2000. 
 
\bibitem{Petzsche88} H.-J.~Petzsche, \emph{On E. Borel's theorem,} Math. Ann. \textbf{282} (1988), 299--313.

\bibitem{Petzsche84} H.-J.~Petzsche, D.~Vogt, \emph{Almost analytic extension of ultradifferentiable functions and the boundary values of holomorphic functions,} Math. Ann. \textbf{267} (1984), 17--35.

\bibitem{Rainer-S} A.~Rainer, G.~Schindl, \emph{Composition in ultradifferentiable classes,} Studia Math. \textbf{224} (2014), 97--131.

\bibitem{RD1969} B.~C.~Rennie, A.~J.~Dobson, \emph{On Stirling numbers of the second kind,} J. Combinatorial Theory \textbf{7} (1969), 116--121.

\bibitem{Roumieu1963}C.~Roumieu, \emph{Ultra-distributions d\'{e}finies sur $\mathbb{R}^{n}$ et sur certaines classes de vari\'{e}t\'{e}s diff\'{e}rentiables,} J. Analyse Math. \textbf{10} (1962/1963), 153--192. 

\bibitem{Rudin62} W.~Rudin, \emph{Division in algebras of infinitely differentiable functions,} J. Math. Mech. \textbf{11} (1962), 797--809. 
\bibitem{Schaefer} H.~H.~Schaefer, \emph{Topological vector spaces}, Springer, New York, 1971. 

\bibitem{Se} R.~T.~ Seeley, \emph{Spherical harmonics}, Amer. Math. Monthly {\bf 73} (1966), 115--121. 

\bibitem{V-E2010} J.~Vindas, R.~Estrada, \emph{On the support of tempered distributions,} Proc. Edinb. Math. Soc. \textbf{53} (2010), 255--270. 



\bibitem{Dorde-V2016} \DJ.~Vu\v{c}kovi\'{c}, J. Vindas, \emph{Eigenfunction expansions of ultradifferentiable functions and ultradistributions in $\mathbb{R}^{n}$,} J. Pseudo-Differ. Oper. Appl. \textbf{7} (2016), 519--531.

\end{thebibliography}
\end{document}